\documentclass[a4paper,11pt]{article}
\usepackage{amscd,amssymb,amsthm}
\usepackage{amssymb,palatino}
\usepackage{euscript}

\newtheorem{theorem}{Theorem}
\newtheorem{lemma}{Lemma}

\newtheorem{proposition}{Proposition}

\newcommand{\sty}{\displaystyle}

\newcommand{\legendre}[2]{( #1 | #2 )}
\begin{document}

%\title{Note on an Additive Characterization of Fibers \\
  %  of the Quadratic Character on $\mathbb F^*_{p^m}$
\title{On a problem  of Perron 
     }
     
\author{Michele Elia}

\date{}

\author{Michele Elia\thanks{The author is
 with the Politecnico di Torino, I-10129
Torino, Italy (e-mail: michele.elia@polito.it)} ~~~}

% \date{December, 2018}
\maketitle

\begin{abstract}
\noindent 
It is shown that a partition $\mathfrak A\cup \mathfrak B$ of the set 
$\mathbb F_{p^m}^*=\mathbb F_{p^m}-\{0 \}$, with $|\mathfrak A|=|\mathfrak B|$,  
is the separation into squares and non squares, 
  if and only if the elements of $\mathfrak A$ and $\mathfrak B$ satisfy certain additive properties, 
  thus providing a purely additive characterization of the set of squares in $\mathbb F_{p^m}$.
\end{abstract} 

\bigskip

\noindent {\bf Mathematics Subject Classification (2010)}: 11A15, 11N69, 11R32

\vspace{2mm} \noindent {\bf Keywords}:  
 finite field, multivariate polynomial, even partition. 

\section{Introduction}
In 1952, Oskar Perron gave some additive properties of the fibers of the quadratic character on $\mathbb F^∗_p$  \cite{perron}.
 Specifically, he showed that if $\mathfrak A,\mathfrak B \subset \mathbb F_p$ are the subsets of quadratic residues and non-residues respectively, then,
letting $d_p=\frac{p-1}{4}$ if $p\equiv 1 \pmod 4$, and $d_p=\frac{p+1}{4}$ if $p\equiv 3 \pmod 4$, 
\begin{enumerate}
 \item Every element of $\mathfrak A$ [respectively $\mathfrak B$] can be written as a sum of two elements of $\mathfrak A$ [respectively $\mathfrak B$] in exactly $d_p-1$  ways.
 \item Every element of $\mathfrak A$ [respectively $\mathfrak B$] can be written as a sum of two elements of $\mathfrak B$ [respectively $\mathfrak A$] in exactly $d_p$ ways.
\end{enumerate}
It was natural to inquire about just how strong this result is, and in \cite{monico} it is shown that these additive properties uniquely characterize the even partition of $\mathbb F^∗_p$ into quadratic residues and non-residues. 
Further, in \cite{monico1} this result has been generalized  (and the "even" restriction removed) to fibers of arbitrary characters on $\mathbb F^∗_p$, with suitable 
cyclotomic numbers in place of the constants above. 
In the present paper, it is considered
the generalization of the even partition (i.e. by the quadratic character $\chi_2$)
 to every finite field of odd characteristic,
 that is, the partition of  $\mathbb F^∗_{p^m}$ into squares and non-squares. 
Specifically, the following theorem is proved.
\begin{theorem}\label{uuu}
Let $p$ be an odd prime, $m$ any positive integer, and set
$$  d_{p^m}=  \left\{ \begin{array}{ll}
           \sty   \frac{p^m-1}{4} ~~, &~~  \mbox{if} ~~  p=1 ~\bmod 4   \\
                         &   \\
           \sty   \frac{p^m-(-1)^m}{4} ~~,&~~  \mbox{if} ~~  p=3~ \bmod 4~~,
       \end{array}  \right.
$$
then 
\begin{description}
\item[1.] Every square $\beta \in \mathbb F^∗_{p^m}$ [non-square $\eta$], i.e. $\chi_2(\beta)=1$  [$\chi_2(\eta)=-1$], can be written as a sum of two squares [non-squares] in exactly  $d_{p^m} - 1$ ways. \\
Every  square [non-square] can be written as a sum of two  non-squares [squares] in exactly $d_{p^m}$ ways. \\
Every non-zero square can be written as a sum of a square and a non-square in exactly $p^m-1-2d_{p^m}$ ways. 
\item[2.] Conversely, every partition of $\mathbb F^*_{p^m}$ into two sets
of the same cardinality and satisfying the above properties, is necessarily constituted by the partition in squares and non-squares as defined by $\chi_2$.
\end{description}
\end{theorem} 

\vspace{4mm}
\noindent
Let $\gamma$ denote a root of an irreducible (possibly primitive) polynomial  $p(x) = x^m+ a_{m-1}x^{m-1}+ \cdots + a_0$  
 over $\mathbb F_p$. The elements of
 $\mathbb F_{p^m}$ are represented, in the basis 
 $\mathcal B=\{ 1,\gamma, \ldots, \gamma^{m-1} \}$, as $m$-dimensional vectors
 in $\mathbb F_p^m$, that is
$$   \beta \in \mathbb F_{p^m}  \Leftrightarrow [b_1, b_2, \ldots b_{m} ] \in \mathbb F_{p}^m ~~, $$
and every element of $\mathbb F_{p^m}$ will be interchangeably denoted with $\beta$ or $[b_1, b_2, \ldots b_{m} ]$. \\
Let $\mathfrak Q_{p^m}$ and $\mathfrak N_{p^m}$ be the sets of squares and
 non-squares of $\mathbb F_{p^m}^*$, respectively; recalling that $\beta \in \mathfrak Q_{p^m}$ 
if and only if the field norm $\mathcal N(\beta)$ is a quadratic residue in $\mathbb F_p$, the quadratic character $\chi_2$ over $\mathbb F_{p^m}$ can be defined,  for $\beta \neq 0$, as 
$\chi_2(\beta)=\legendre{\mathcal N(\beta)}{p}$, where $\legendre{.}{p} $ indicates the Legendre symbol.
To prove  Theorem \ref{uuu}, we introduce a description of $\mathfrak Q_{p^m}$ and $\mathfrak N_{p^m}$
using multivariate polynomials in $m$ variables
$$  \mathfrak Q_{p^m}  \Leftrightarrow   q_r(\mathbf x)= \sum_{\beta \in \mathfrak Q_{p^m}} \prod_{j=1}^{m} x_j^{b_j}   ~~~,~~~
      \mathfrak N_{p^m}   \Leftrightarrow q_n(\mathbf x)=  \sum_{\beta \in \mathfrak N_{p^m}} \prod_{j=1}^{m} x_j^{b_j}  ~~.  $$

\subsection{Proof of Claim 1.}
     \label{sect1}
Since the fibers  $ \mathfrak Q_{p^m}$ and $ \mathfrak N_{p^m}$ form a partition of $\mathbb F_{p^m}$,
 the $0 \in \mathbb F_{p^m}$ standing alone, we have
$$    1 + q_r(\mathbf x) + q_n(\mathbf x) =  \prod_{j=1}^{m} \frac{x_j^p-1}{x_j-1}  ~~,   $$
as well as the following proposition

\begin{proposition}
   \label{pippo}
The set $\{1,  q_r(\mathbf x), q_n(\mathbf x) \}$ is a basis of a subspace $\mathcal V_{3}$ of dimension $3$ in the $p^m$-dimensional vector space 
$\mathbb Q[\mathbf x]/\langle (x_1^p -1) , (x_2^p -1),  \cdots,  (x_m^p -1)  \rangle$ of multivariate polynomials of degree at most $p -1$ in each variable.
\end{proposition}

\noindent
The representatives of $q_r(\mathbf x)^2$  and $q_n(\mathbf x)^2$ modulo 
 $\langle (x_1^p -1) , (x_2^p -1),  \cdots,  (x_m^p -1)  \rangle$ in $\mathbb Q[x]$
 are denoted by
$$  \begin{array}{lclcl}
q_r(\mathbf x)^2& = &  \sum_{\beta \in \mathfrak Q_r} A_{b_1, \ldots,b_m} \prod_{j=1}^{m} x_j^{b_j}&~~&  \bmod \langle (x_1^p -1) , (x_2^p -1),  \cdots,  (x_m^p -1)  \rangle  \\
 q_n(\mathbf x)^2& = &  \sum_{\beta \in \mathfrak Q_n}B_{b_1, \ldots,b_m} \prod_{j=1}^{m} x_j^{b_j} &~~&  \bmod \langle (x_1^p -1) , (x_2^p -1),  \cdots,  (x_m^p -1)  \rangle  \\
      \end{array}
$$
where $ A_{b_1, \ldots,b_m}$ and  $B_{b_1, \ldots,b_m}$ are non-negative integers smaller than $p^m$.   
It is observed that $ A_{b_1, \ldots,b_m} $ [or $B_{b_1, \ldots,b_m}$] is precisely the number of ways in which every $\beta \in \mathbb F_{p^m}$ can be written as a sum of two 
squares [or non-squares]. 
The numbers
 $ A_{b_1, \ldots,b_m}$ and $B_{b_1, \ldots,b_m}$ are elements of the set $\mathcal R= \{0,1,2,\ldots ,p^m-1\}$.
% of canonical representatives of the ring $\mathbb Z/p^m \mathbb Z$. %can be considered as

\begin{lemma}
  \label{lem11}
 Let $p$ be an odd prime, $m$ be a positive integer, and $ A_{b_1, \ldots,b_m}  , B_{b_1, \ldots,b_m}$ as defined above. 
Then for every $\alpha,  \beta \in  \mathbb F_{p^m}$,
the following hold:
 \begin{enumerate}
    \item  $ B_{b_1, \ldots,b_m}-A_{b_1, \ldots,b_m}=(\mathcal N(\beta)|p)$.
    \item  If $(\mathcal N(\beta)|p)=(\mathcal N(\alpha)|p)$, then 
      $A_{b_1, \ldots,b_m} =A_{a_1, \ldots, a_m}$   and   
      $B_{b_1, \ldots,b_m} =B_{a_1, \ldots, a_m}$.
   \item If $(\mathcal N(\beta)|p)\neq (\mathcal N(\alpha)|p)$, then
$$   \begin{array}{l}
         A_{b_1, \ldots,b_m} =A_{a_1, \ldots, a_m}+(\mathcal N(\alpha)|p)   \\
        B_{b_1, \ldots,b_m} =B_{a_1, \ldots, a_m}-(\mathcal N(\alpha)|p)   \\
       \end{array}
$$	
 \end{enumerate}
\end{lemma}

\begin{proof}
Let $\mathbf e$ be the all-one $m$-dimensional vector, then $q_r(\mathbf e)=q_n(\mathbf e) =\frac{p^m-1}{2}$, 
   thus 
$$q_r(\mathbf x)-q_n(\mathbf x) = Q(\mathbf x) \prod_{j=1}^{m}( x_j-1) ~~ $$
which, multiplied by
$$     q_r(\mathbf x) + q_n(\mathbf x) = -1 + \prod_{j=1}^{m} \frac{x_j^p-1}{x_j-1}  ~~,   $$
gives
$$ q_r(\mathbf x)^2-q_n(\mathbf x)^2 = -Q(\mathbf x) \prod_{j=1}^{m}( x_j-1) + Q(\mathbf x) \prod_{j=1}^{m} (x_j^p-1) =  -Q(\mathbf x) \prod_{j=1}^{m}( x_j-1)~~ \bmod \prod_{j=1}^{m} (x_j^p-1)  ~~. $$
That is
$$  q_n(\mathbf x)^2-q_r(\mathbf x)^2 = q_r(\mathbf x)-q_n(\mathbf x)    \Rightarrow 
      q_r(\mathbf x)^2 + q_r(\mathbf x) =   q_n(\mathbf x)^2 + q_n(\mathbf x)  \pmod {\prod_{j=1}^{m} (x_j^p-1) } ~~, $$
where the left equality proves item 1. \\
Suppose now that $\chi_2(\alpha)=\chi_2(\beta)=1$  in $ \mathbb F_{p^m}$. Then there exist a square
 $\delta \in   \mathbb F_{p^m}$ so that $\beta=\delta \alpha$. 
 If $\chi_2(x)=\chi_2(y)=1$, with $\alpha = x+y$, it follows that $\beta = \delta x+\delta y$ and 
 $\delta x, \delta y$ are also squares. 
Thus $A_{b_1, \ldots,b_m} =A_{a_1, \ldots, a_m}$, and with a similar argument
 $B_{b_1, \ldots,b_m} =B_{a_1, \ldots, a_m}$. \\
Suppose $\chi_2(\alpha)=1$, and that $\alpha = x+y$  is a sum of two non-squares, 
 let $\beta$ be any non-square, then
$$ \eta=   \beta  \alpha = \beta x +\beta y  $$
says that a non-square is the sum of two squares, it follows that  $A_\eta= B_\alpha$ with $\eta$ a non-square,
 and $\alpha$ a square, 
the same equality holds by exchanging square and non-square.
\end{proof}

\noindent
Let $A_1$ and $A_{-1}$ denote the common value of the $A_{\alpha}$ with 
$\legendre{\mathcal N(\alpha)}{p} = 1$ and $-1$, respectively. 
Similarly, define $B_1$ and $B_{-1}$  to be the common values of the $B_{\alpha}$ for $\legendre{\mathcal N(\alpha)}{p} = 1$ and $-1$, respectively. 
Further, from Lemma \ref{lem11}, we have $A_1=B_{-1}$, and
$B_1=A_{-1}$. Let $A_0$ denote the number of sums of two squares giving $0$,
then $A_0=0$ if $p\equiv 3 \pmod 4$ and $m$ odd because $\chi_2(-1)=-1$, otherwise $A_0=\frac{p^m-1}{2}$ because $\chi_2(-1)=1$, i.e. $-1$ is a square.
A direct counting of the number of sums of two squares gives
$$ \frac{p^m-1}{2} A_1+ \frac{p^m-1}{2} A_{-1} +A_0 =\left(\frac{p^m-1}{2} \right)^2~~, $$
therefore, in view of the above observations, we have
\begin{equation}
   \label{zero}
 A_1+A_{-1}  = \left\{ \begin{array}{lcl}
                                         \sty   \frac{p^m-3}{2} &~~&  \mbox{if  $p\equiv 1 \pmod 4$}  \\
          & & \\
                                       \sty     \frac{p^m-2-(-1)^m}{2} &~~&  \mbox{if  $p\equiv 3 \pmod 4$}  \\
          \end{array}   \right. ~~,
 \end{equation}
furthermore $A_1+A_{-1} =B_1+B_{-1}$.

\begin{theorem} 
   \label{th2}
Let $p^m$ be a power of an odd prime and set
\begin{equation}
   \label{eqdp}
  d_{p^m} = \left\{  \begin{array}{lcl}
       \sty  \frac{p^m-1}{4} & ~~& \mbox{if  $p \equiv 1 \pmod{4}$} \\
           &  &  \\
        \sty  \frac{p^m-(-1)^m}{4} &~~& \mbox{if  $p \equiv 3 \pmod{4}$} ~. \\
       \end{array}  \right.
\end{equation}
Then every square [non-square] can be written as a sum of two squares [non-squares] in exactly 
$d_{p^m} - 1$ ways. Every square [non-square] can be written as a sum of two non-squares in exactly 
$d_{p^m}$ ways. Moreover, every non-zero element can be written as a sum of a square and a 
non-square in exactly $p^m-1-2d_{p^m}$ ways.
\end{theorem}

\begin{proof}
As above, let $A_1$, $B_1$ denote respectively the number of ways in which a square can be written 
as a sum of two squares or two non-squares. Let $A_{-1}$, $B_{-1}$ denote respectively 
the number of ways in which a non-square can be written as a sum of two squares or two 
 non-squares. 
To show that $d_{p^m}  =A_{-1} = B_{1}$ and  $d_{p^m} - 1 = A_{1} = B_{-1}$, we 
consider the equation (\ref{zero}) and the equation $A_{-1}-A_1=1$ from Lemma \ref{lem11},
which is obtained noting that $A_{-1}=B_{1}$. \\
Solving the trivial linear system we obtain the claimed values for $A_1$, $A_{-1}$, and in turn the values for $B_1$, $B_{-1}$. \\
The number of ways, that every non-zero element is written as a sum of a square and a 
non-square, is obtained by observing that the equation $x_1 + x_2 = \beta \neq 0$ in 
$\mathbb F_{p^m}$ has $p^m - 2$ solutions with neither $x_1$ nor $x_2$ equal $0$. 
Therefore, the number of solutions with $x_1$ a square, and $x_2$ a 
non-square, or vice-versa, is $p^m - 2 - (2 d_{p^m} - 1)$.
\end{proof}

\subsection{Proof of Claim 2.}  %, (the converse)}
The goal of this section is to show that the additive properties given in Section \ref{sect1}
 completely characterize the squares in $\mathbb F_{p^m}$. 
Let $d_{p^m}$ be defined as in Equation (\ref{eqdp}), and, for the remainder of this section, 
suppose $\mathfrak A$ and $\mathfrak B$ form an even partition of 
$\mathbb F_{p^m}  - \{ 0 \}$ such that

\begin{enumerate}
   \item Every element of $\mathfrak A$ [$\mathfrak B$] can be written as a sum of two elements
      from
      $\mathfrak A$ [$\mathfrak B$] in exactly $d_{p^m} -1$ ways.
\item Every element of $\mathfrak A$ [ $\mathfrak B$ ] can be written as a sum of two elements
    from $\mathfrak B$ [$\mathfrak A$ ] in exactly $d_{p^m}$ ways.
\end{enumerate}

\noindent
Define two polynomials in  $\mathbb F_{p^m}[\mathbf x]$,
$$   a(\mathbf x)= \sum_{\beta \in \mathfrak A} \prod_{j=1}^{m} x_j^{b_j}   ~~~,~~~
     b(\mathbf x)=  \sum_{\beta \in \mathfrak B} \prod_{j=1}^{m} x_j^{b_j}   ~~.  $$
￼
It follows from the assumptions on the sets  $\mathfrak A$ and $\mathfrak B$  that
$$  \begin{array}{lcl}
          a(\mathbf x)^2 &=&  (d_{p^m} -1)a(\mathbf x) + d_{p^m} b(\mathbf x) + c_ {p^m}  \\
          &=& d_{p^m}( -1 + \prod_{j=1}^{m} \frac{x_j^p-1}{x_j-1})-a(\mathbf x)+c_{p^m} \\
        &=& d_{p^m}( -1 + \prod_{j=1}^{m} (x_j-1)^{p-1})-a(\mathbf x)+
   c_{p^m} ~~\bmod \prod_{j=1}^{m} (x_j^p-1) , \pmod p    ~~.
 \end{array}
   $$
where the identity $\frac{x_j^p-1}{x_j-1}= (x_j-1)^{p-1} \bmod p$ has been used.
Thus, we can write the equation 
\begin{equation}
   \label{eqfund1}
 a(\mathbf x)^2 + a(\mathbf x) = d_{p^m}(\prod_{j=1}^{m} (x_j-1)^{p-1} -1) + c_{p^m}
     ~~\pmod{\prod_{j=1}^{m} (x_j^p-1) }~,~\pmod p  ~~,
\end{equation}
where
\begin{enumerate}
  \item[] $c_{p^m}$ is the number of ways in which zero can be written as a sum of two elements of
 $\mathfrak A$, and can be explicitly computed by evaluating (\ref{eqfund1}) at 
 $\mathbf x = \mathbf e$, obtaining $c_{p^m}=\frac{p^m-1}{2}$ if  $p=1 \pmod 4$, and 
 $c_{p^m}=0$ if  $p=3 \pmod 4$ and $m$ is odd.
%  \item  $\frac{x_j^p-1}{x_j-1} = (x_j-1)^{p-1} ~~ \bmod p$, for every $j$.
\end{enumerate}
Similarly, we find 
$$ b(\mathbf x)^2 + b(\mathbf x) = d_{p^m}(\prod_{j=1}^{m} (x_j-1)^{p-1} -1) +  c_{p^m}
     ~~\pmod{\prod_{j=1}^{m} (x_j^p-1) }  ~~. $$
To show that 
$\{a(\mathbf x), b(\mathbf x)\} = \{q_r(\mathbf x), q_n(\mathbf x)\}$, it is sufficient to show
the coincidence modulo $p$, since the coefficients of every polynomial are $0$ or $1$, we proceed as follows.

\begin{lemma}
  \label{lemhensel}
Let $p$ be an odd prime, and 
  $\mathbb R_k = \mathbb F_{p^m}[x]/\langle {\prod_{j=1}^{m} (x_j-1)^k }\rangle$
  for $k \geq 1$. 
 Then each invertible element of $\mathbb R_k$ has at most two distinct square roots.
\end{lemma}

\begin{proof}
We proceed by induction on $k$. The initial case is obvious since
 $\mathbb R_1 = \mathbb F_{p^m}$. Suppose now that the result holds for all 
  $1 \leq k \leq N$. Further suppose that the polynomials $a(\mathbf x),b(\mathbf x),c(\mathbf x),
    g(\mathbf x) \in  \mathbb F_{p^m}[\mathbf x]$ are invertible 
   modulo  $\langle \prod_{j=1}^{m} (x_j-1)^{N+1} \rangle$ and
$$ a(\mathbf x)^2 +\langle \prod_{j=1}^{m} (x_j-1)^{N+1} \rangle= b(\mathbf x)^2 +\langle \prod_{j=1}^{m} (x_j-1)^{N+1} \rangle =
 c(\mathbf x)^2 +\langle \prod_{j=1}^{m} (x_j-1)^{N+1} \rangle = g(\mathbf x)^2 +\langle \prod_{j=1}^{m} (x_j-1)^{N+1} \rangle ~~.  $$
By canonical projection onto $\mathbb R_N$, it follows that 
$$ a(\mathbf x)^2 +\langle \prod_{j=1}^{m} (x_j-1)^{N} \rangle= b(\mathbf x)^2 +\langle \prod_{j=1}^{m} (x_j-1)^{N} \rangle =
 c(\mathbf x)^2 +\langle \prod_{j=1}^{m} (x_j-1)^{N} \rangle = g(\mathbf x)^2 +\langle \prod_{j=1}^{m} (x_j-1)^{N} \rangle ~~,  $$
so that two of these must be equal by the induction hypothesis, say 
$$  a(\mathbf x) +\langle \prod_{j=1}^{m} (x_j-1)^{N} \rangle= b(\mathbf x) +\langle \prod_{j=1}^{m} (x_j-1)^{N} \rangle ~~. $$
 It follows that $a(\mathbf x) =  b(\mathbf x) + \prod_{j=1}^{m} (x_j-1)^{N} f(\mathbf x)$ for some 
$f(\mathbf x) \in  \mathbb F_{p^m}[x]$. 
Thus
$$  \begin{array}{lcl}
          b(\mathbf x)^2 +\langle \prod_{j=1}^{m} (x_j-1)^{N+1} \rangle &=& 
                        a(\mathbf x)^2 +\langle \prod_{j=1}^{m} (x_j-1)^{N+1} \rangle \\
                 &=&(b(\mathbf x) + \prod_{j=1}^{m} (x_j-1)^{N}f(\mathbf x) )^2+
                            \langle \prod_{j=1}^{m} (x_j-1)^{N+1} \rangle \\
             &=& b(\mathbf x)^2+2b(\mathbf x)f(\mathbf x) \prod_{j=1}^{m} (x_j-1)^{N} + \\
              &~ & \hspace{20mm}  f(\mathbf x)^2   \prod_{j=1}^{m} (x_j-1)^{2N} +
                            \langle \prod_{j=1}^{m} (x_j-1)^{N+1} \rangle \\  
             &=&   b(\mathbf x)^2+2b(\mathbf x)f(\mathbf x) \prod_{j=1}^{m} (x_j-1)^{N} +
                            \langle \prod_{j=1}^{m} (x_j-1)^{N+1} \rangle. \\  
    \end{array}
$$
So $2( \prod_{j=1}^{m} (x_j-1)^{N})b(\mathbf x) f(\mathbf x) \in \langle \prod_{j=1}^{m} (x_j-1)^{N+1} \rangle$, but since 
 $2b(\mathbf x)$ is invertible modulo $\langle \prod_{j=1}^{m} (x_j-1)^{N} \rangle$, it follows that
 $\prod_{j=1}^{m} (x_j-1) | f(\mathbf x) $, so that 
$a(\mathbf x)+\langle \prod_{j=1}^{m} (x_j-1)^{N+1} \rangle= b(\mathbf x) + \langle \prod_{j=1}^{m} (x_j-1)^{N+1} \rangle$.
\end{proof}

\begin{theorem}
   \label{th1inv}
Let $p$ be an odd prime and let $d_{p^m}$ be defined as in Equation (\ref{eqdp}). Suppose 
$\mathfrak A \in \mathbb F_{p^m}^*$ and 
$\mathfrak B =  \mathbb F_{p^m}^* \backslash \mathfrak A$.
 Then $\mathfrak A$ is precisely the set of squares of $ \mathbb F_{p^m}^*$ if and only if
\begin{enumerate}
   \item $|\mathfrak A|= \frac{p^m-1}{2}$,
   \item  $1 \in \mathfrak A$,
   \item Every element of $\mathfrak A$ can be written as a sum of two elements from $\mathfrak A$
            in exactly $d_{p^m} - 1$ ways.
   \item Every element of $\mathfrak B$ can be written as a sum of two elements from
            $\mathfrak A$ in exactly $d_{p^m}$ ways.
\end{enumerate}
\end{theorem}

\begin{proof}
As in Equation (\ref{eqfund1}), it follows from the hypotheses that
$$ a(\mathbf x)^2 + a(\mathbf x) = d_{p^m} (\prod_{j=1}^{m} (x_j-1)^{p-1} -1) + c_{p^m}
     ~~\pmod{\prod_{j=1}^{m} (x_j^p-1) }  ~,$$
where
$$c_{p^m} =  \left\{  \begin{array}{lcl}
     \frac{p^m-1}{2}~,&~~& \\ % \mbox{if}~~ p \equiv  1   \pmod 4  \\
      0~,                 &~~&  \mbox{if}~~ p \equiv  3   \bmod 4 \bigwedge m\equiv 1 \bmod 2 ~~.
   \end{array}  \right.
$$
It is an immediate corollary of Lemma \ref{lemhensel} that a quadratic equation in 
$\mathbb R_k[y]$ with invertible coefficients has at most two solutions (this follows
 from a completing-the-square argument). In particular, the equation 
$$ y^2+y - d_{p^m} (\prod_{j=1}^{m} (x_j-1)^{p-1} -1)-c_{p^m}= 0 $$
 has coefficients invertible  in 
$\mathbb R_p = \mathbb F_{p^m} [\mathbf x]/ \langle {\prod_{j=1}^{m} (x_j-1)^p }\rangle$
so that it has at most two distinct solutions in $\mathbb R_p$. ~
From the proof of Lemma \ref{lem11}, we have that $q_r(\mathbf x)$ and $q_n(\mathbf x)$
 are two distinct solutions, so that $a(\mathbf x) = q_r(\mathbf x)$ or $a(\mathbf x) = q_n(\mathbf x)$. 
But since $1 \in \mathfrak A$ and $\mathfrak A, \mathfrak B$ are disjoint by assumption,
 it must be the case that ~
 $a(\mathbf x) = q_r(\mathbf x)$.

\end{proof}

% *****************************************************************************

%

\begin{thebibliography}{99}
\bibitem{albert}
     A.A. Albert,
    {\em Structure of Algebras}, AMS, Providence, R.I. 2003.
%
\bibitem{dickson}
    L.E. Dickson, {\em Algebras and their Arithmetics},
     Dover, New York, NY 1960.
%
\bibitem{monico}
    C. Monico, M. Elia, Note on an Additive Characterization of Quadratic Residues Modulo $p$,
    {\em Journal of Combinatorics, Information \& System Sciences}, 
     31 (2006), 209-215.
%
\bibitem{monico1}
    C. Monico, M. Elia, An Additive Characterization of Fibers of Characters on $\mathbb F^m��_p$,
    {\em International Journal of Algebra}, 
     Vol. 1-4, n.3, 2010, p.109-117.
%
\bibitem{perron}
    O. Perron, Bemerkungen u ̈ber die Verteilung der quadratischen Reste, {\em Mathematische
       Zeitschrift}, 56(1952), 122-130.
%
\bibitem{everest}
    A. Winterhof, On the Distribution of Powers in Finite Fields, 
     {\em Finite Fields and their Applications}, 4(1998), 43-54.
%
\end{thebibliography}
\end{document}